\documentclass[11pt]{amsart}
\theoremstyle{plain}
\newtheorem*{theorem*}{Theorem}
\theoremstyle{definition}
\newtheorem*{remark*}{Remark}
\begin{document}
\title{A note on a conjecture by Ulas on polynomial substitutions}
\author{Peter M\"uller}
\address{Institut f\"ur Mathematik\\
  Universit\"at W\"urzburg\\
  97074 W\"urzburg\\
  Germany}
\email{peter.mueller@mathematik.uni-wuerzburg.de}
\begin{abstract}
We prove a recent conjecture by Ulas on reducible polynomial substitutions.
\end{abstract}
\maketitle
We prove the following result, which was conjectured by Ulas in
\cite{ulas___polysubs} and proven for $d\le4$ there.
\begin{theorem*}
  Let $f(X)$ be an irreducible polynomial of degree $d\ge3$ over a
  field $K$. Then there is a polynomial $h(X)\in K[X]$ of degree $\le
  d-1$ such that $f(h(X))$ is reducible over $K$.
\end{theorem*}
\begin{proof}
  Let $\alpha$ be a root of $f(X)$ in some extension of $K$, and
  $g(X)$ be the minimal polynomial of $\frac{1}{\alpha}$ over $K$. (So
  $g(X)$ is, up to a non-zero factor from $K$, the reciprocal of
  $f(X)$.) Since $K(\frac{1}{\alpha})=K(\alpha)$, and
  $\frac{1}{\alpha}$ has degree $d$ over $K$, we have
  $\alpha=h(\frac{1}{\alpha})$ for some $h(X)\in K[X]$ of degree at
  most $d-1$.

We obtain $f(h(\frac{1}{\alpha}))=f(\alpha)=0$. So $f(h(X))$ shares
the root $\frac{1}{\alpha}$ with the irreducible polynomial $g(X)$,
thus $g(X)$ divides $f(h(X))$. The degree of $h$ is at least $2$, for
otherwise $\alpha$ would have degree at most $2$ over $K$, contrary to
$d\ge3$.

The assertion now follows from $\deg f(h(X))\ge2d$ and $\deg g(X)=d$.
\end{proof}
\begin{remark*}
  (a) Of course one can formulate the proof without reference to the
  algebraic element $\alpha$. Suppose without loss that $f(X)$ is
  monic. Set $g(X)=X^df(1/X)$ and $h(X)=(1-g(X))/X$. Then $h(X)$ is a
  polynomial, and $g(X)$ divides $f(h(X))$. This example is similar to
  the one by Schinzel in \cite[Lemma 10]{schinzel___gelfond}, which,
  as pointed out by Ulas \cite[page 59]{ulas___polysubs}, proves the
  theorem provided that $d-1$ does not divide the characteristic of
  $K$.

  (b) The answer to \cite[Question 5.5]{ulas___polysubs}, a kind of
  converse to the above theorem, is negative if $K=\mathbb R$ by lack
  of irreducible polynomials of degree $\ge3$.
\end{remark*}
\providecommand{\bysame}{\leavevmode\hbox to3em{\hrulefill}\thinspace}

\end{document}